\documentclass[a4paper,reqno,12pt]{amsart}
\usepackage[dvips]{graphicx}
\usepackage{graphics}
\usepackage{amsthm, amssymb, amsmath, amsfonts}

\newtheorem{thm}{Theorem}[section]
\newtheorem{lem}{Lemma}[section]

\newtheorem{defn}{Definition}[section]
\newtheorem{ex}{Example}[section]
\newtheorem{rmk}{Remark}[section]

\newcommand{\Real}{{\mathbb R}}
\newcommand{\Natural}{{\mathbb N}}
\newcommand{\ds}{\displaystyle}
\newcommand{\To}{\longrightarrow}
\def\1{\'{a}}
\def\2{\'{e}}
\def\3{\'{o}}
\def\4{\'{u}}
\def\5{\H{o}}
\def\6{\H{u}}
\def\a{\alpha}

\def\l{\lambda}
\def\<{\langle}
\def\>{\rangle}

\begin{document}
\title{Existence of solutions of inverted variational inequalities}
\author{ Szil\'ard L\' aszl\' o$^*$}
\thanks{$^*$This work was supported by a grant of the Romanian National Authority for Scientific Research, CNCS-UEFISCDI, project number PN-II-ID-PCE-2011-3-0024.}
\address{S. L\' aszl\' o, Department of Mathematics, Technical University of Cluj Napoca,
Str. Memorandumului nr. 28, 400114 Cluj-Napoca, Romania.}
\email{szilard.laszlo@math.utcluj.ro}

\begin{abstract} In this paper we introduce two new generalized variational inequalities, and we give some existence results of the solutions for these variational inequalities involving operators belonging to a recently introduced class of operators. We show by examples, that our results fail outside of this class. Further, we establish a result that may be viewed as a generalization of Minty's theorem, that is, we show that under some circumstances the set of solutions of these variational inequalities coincide. We also show that the operators, involved in these variational inequalities, to belong to the above mentioned class, is essential in obtaining this result.
 As application, we show that Brouwer's fixed point theorem is an easy consequence of our results.
\newline

\emph{Keywords:}{Operator of type ql; variational inequality; Minty type theorem; KKM mapping}

\emph{MSC:} 26A48, 47H99, 54A99
\end{abstract}

\maketitle

\section{Introduction}
 The theory of variational inequalities-often related to the theory of (generalized) monotone operators-goes back to Stampacchia (see \cite{Sta}), and Fichera (see \cite{Fic}), and provides powerful techniques for studying problems arising in optimization, transportation, economic equilibrium, and other problems of practical interest. For instance, the free and moving boundary value problems can be studied effectively in the framework of variational inequalities (see, for instance, \cite{Bai, Ber-Gaf, Daf}). For a comprehensive discussion on variational inequalities and their use see see \cite{Kin-Sta}. However, the classical variational inequality theory developed so
far is applicable for studying free and moving boundary value problems of even order only. In recent years, many generalizations of variational inequalities have been considered, and it has been shown that generalized variational inequalities provide us with an unified, simple, and natural framework to study a wide class of problems including unilateral, moving, obstacle, free, equilibrium, and economics arising in various areas of pure and applied sciences.

In this paper we give some existence results of the solutions for some generalized variational inequalities.

Throughout this paper, unless is otherwise specified, we assume that $X$ is a real Banach space and $X^*$ is the topological dual of $X.$ We denote by $\<x^*,x\>$  the value of the linear and continuous functional $x^*\in X^*$ in $x\in X.$ Consider the set $K\subseteq X,$ and let  $A:K\To X^*$ and $a:K\To X$ be two given operators. The general variational inequalities, considered in this paper, are some generalizations of the classic  variational inequalities of Stampacchia and of Minty.

Recall that the so called general variational inequality of Stampacchia type, $VI_S(A,a,K)$, (see \cite{La}),  consists in finding an element $x\in K,$ such that
\begin{equation}\label{e4.1.1}
\<A(x),a(y)-a(x)\>\ge 0,\,\mbox{\, for all\,}\,y\in K,
\end{equation}
Obviously, when $a\equiv id_K,$ then $(\ref{e4.1.1}),$ reduces to the well known Stampachia variational inequality $VI_S(A,K)$ which consists in finding an element $x\in K,$ such that
\begin{equation*}
\<A(x),y-x\>\ge 0,\,\mbox{\, for all\,}\,y\in K.
\end{equation*}
Interchanging the role of $A$ and $a$ in $VI_S(A,a,K),$ we obtain the \emph{inverted general variational inequality  of Stampacchia type}, $VI_{iS}(A,a,K),$ which consist of finding an element  $x\in K$ such that
\begin{equation}\label{e4.1.2}
\<A(y)-A(x),a(x)\>\ge 0,\,\mbox{\, for all\,}\,y\in K.
\end{equation}
\begin{rmk}\rm\label{r4.1.1}
It can be easily observed that in Hilbert spaces, (i.e $X$ is a Hilbert space), the formulations $(\ref{e4.1.1})$ and $(\ref{e4.1.2})$ coincide.
\end{rmk}
Recall that the general variational inequality of Minty type, $VI_M(A,a,K)$, (see \cite {La}), consists in finding an element $x\in K,$ such that
\begin{equation}\label{e4.1.3}
\<A(y),a(y)-a(x)\>\ge 0,\,\mbox{\, for all\,}\,y\in K.
\end{equation}
Obviously, when $a\equiv id_K,$ then $(\ref{e4.1.3})$  reduces to the well known Minty variational inequality $VI_M(A,K)$ which consists in finding an element $x\in K,$ such that
\begin{equation*}
\<A(y),y-x\>\ge 0,\,\mbox{\, for all\,}\,y\in K.
\end{equation*}

Interchanging the role of $A$ and $a$ in $VI_M(A,a,K),$ we obtain the \emph{inverted general variational inequality of Minty type}, $VI_{iM}(A,a,K),$ which consist of finding an element  $x\in K$ such that
\begin{equation}\label{e4.1.4}
\<A(y)-A(x),a(y)\>\ge 0,\,\mbox{\, for all\,}\,y\in K.
\end{equation}
\begin{rmk}\rm\label{r4.1.2}
It can be easily observed that in Hilbert spaces the formulations $(\ref{e4.1.3})$ and $(\ref{e4.1.4})$ coincide.
\end{rmk}
In \cite{La} were presented some existence results for $(\ref{e4.1.1}),$ respectively $(\ref{e4.1.3}),$ in the case when $K$ is weakly compact and convex.

Also in \cite{La}, was introduced a new class of operators, the class of operators of type ql, that seems to be a useful extension of the class of linear operators.  Further, it was shown, that in one dimensional case, this class coincide with the set of monotone functions, even more, an operator taking its values in $\Real$ is of type ql if, and only if, is quasilinear.

The paper is organized as follows. In Section 2, we give some preliminary results that will be used throughout this paper. In Section 3, we state and prove the main results of this paper. Based on the concept of operators of type ql,  some sufficient conditions, which ensure the existence of the solutions for the inverted general variational inequality of Stampacchia type, are provided. Existence results for  problem $(\ref{e4.1.2})$ will be helpful in the sense that one would like to know if a solution of  problem $(\ref{e4.1.2})$ exists before one actually devises some plausible algoritheorems for solving  the problem. Moreover, we furnish  some sufficient conditions which ensure that the set of solutions of the inverted general variational inequalitiy of Stampacchia type and the set of solutions of the inverted general variational inequality of Minty type coincide. In Section 4,  we show that Brower's fixed point theorem is an easy consequence of our results.

\section{Preliminaries}

 In what follows we present the notion of KKM mapping, which was initially introduced by Knaster, Kuratowski and Mazurkiewicz, and will be very useful in establishing the existence of the solutions for the problem $(\ref{e4.1.2})$.

 Let $X$ be a real linear space and let $D\subseteq X.$ Recall that the convex hull of the set  $D$ is defined as the set
\[co (D):=\left\{\sum_{i=1}^n\l_i x_i: x_i\in D,\, \sum_{i=1}^n\l_i=1,\, \l_i\ge 0,\,\forall\, i\in\{1,2,\ldots,n\},\, n\in \Natural\right\}.\]

 \begin{defn}\rm(Knaster-Kuratowski-Mazurkiewicz) Let $X$ be a Hausdorff linear space and let $M\subseteq X.$ The application $G:M\rightrightarrows X$ is called a KKM mapping, if for every finite number of elements $x_1,x_2,\dots,x_n\in M$ one has $co\{x_1,x_2,\ldots,x_n\}\subseteq \bigcup_{i=1}^n G(x_i).$
\end{defn}

The following result is due to  Ky Fan (see \cite{Fan}).
\begin{lem}\label{Fan} Let $X$ be a Hausdorff linear space,  $M\subseteq X$ and $G:M\rightrightarrows X$ be a KKM mapping. If $G(x)$ is closed for every $x\in M$, and there exists $x_0\in M,$ such that $G(x_0)$ is compact, then
$$\bigcap_{x\in M}G(x)\neq\emptyset.$$
\end{lem}

 Let $X$ be a real linear space. For $x,y\in X$ let us denote by $[x,y]:=\{z=(1-t)x+ty:t\in[0,1]\}$ the closed line segment with the endpoints  $x$ respectively $y.$ The open line segment with the endpoints $x$ respectively $y$ is defined by $(x,y):=[x,y]\setminus\{x,y\}=\{z=(1-t)x+ty:t\in(0,1)\}.$ The following concept, was introduced in \cite{La}.

\begin{defn}\rm\label{d1}Let $X$ and $Y$ be two real linear spaces. One says that the operator $A:D\subseteq X\To Y$ is of type ql, if for every $x,y\in D$ and every $z\in [x,y]\cap D$ one has $A(z)\in[A(x),A(y)].$ One says that $A$ is of type strict ql, if for every $x,y\in D,\,x\neq y$ and every $z\in (x,y)\cap D$ one has $A(z)\in(A(x),A(y)).$
\end{defn}

Obviously an operator of type strict ql is also of type ql, but the converse is not true. It can be easily realized, that these concepts are generalizations of the notions of monotonicity, respectively strict monotonicity of  real valued functions of one real variable. It is also easy to observe, that these notions may be viewed as generalizations of the concept of linear operator  as well. For a nontrivial example of operator of type ql see \cite{La}.

 Let $X$ be a real linear space, let $Y$ be a topological space, and let $A:D\subseteq X\To Y$ be an operator. Recall, that we say that $A$ is continuous on line segments at $x\in D,$ if for every sequence  $\{t_n\}\subseteq \Real$ of real numbers convergent to $0,$ and every $y\in D$ with $x+t_ny\in D$, we have $A(x+t_ny)\To A(x),\, n\To\infty.$ We say, that $A$ is continuous on line segments in $D,$ if it has this continuity property at every $x\in D.$ If $X$ is a Banach space and $Y=X^*$ is its dual, then the above defined continuity is called hemicontinuity.
  We will need the following result from \cite{La}.

 \begin{thm}\label{t1}(Theorem 3.2 \cite{La}) Let $X$ and $Y$ be two real linear spaces, let $D\subseteq X$ be convex, and let $A:D\To Y$ be an operator of type ql. Then for every $n\in\Natural,$ every $x_1,x_2,\ldots,x_n\in D,$ and every $x\in co\{x_1,x_2,\ldots,x_n\},$ we have $A(x)\in co\{A(x_1),A(x_2),\ldots,A(x_n)\}.$
\end{thm}

Further, we need the following definitions.

\begin{defn}\rm Let $X$ be a real Banach space, $X^*$ be its topological dual, and let $A:D\subseteq X\To {X^*}$ be an operator. We say that $A$ is monotone (in Minty-Browder sense, see for instance, \emph{\cite{Bro,Min}}) if for every $x,y\in D$ one has $\<A(x)-A(y),x-y\>\ge 0.$ We say that $A$ is pseudomonotone (see $\emph{\cite{Gia-Mau,La1}}$), if $\<A(x),y-x\>\ge0$ implies $\<A(y),y-x\>\ge 0$  for all $x,y\in K.$
\end{defn}

Generalizing these concepts, we have the following definitions (see \cite{Nor2}).
 \begin{defn}\rm  Let $X$ be a real Banach space, $X^*$ be its topological dual, and let $A:D\subseteq X\To {X^*}$ and $a:D\To X$ be given operators.  We say that $A$ is monotone relative to $a,$ if for all $x,y\in D,$ we have $\<A(x)-A(y),a(x)-a(y)\>\ge 0.$
  We say that $A$ is $a$-pseudomonotone, if  $\<A(x),a(y)-a(x)\>\ge0$ implies $\<A(y),a(y)-a(x)\>\ge 0$  for all $x,y\in D.$
\end{defn}
\begin{rmk}\rm\label{r4.1}
Obviously, if $a\equiv id_D$ we obtain the definition of the Minty-Browder monotonicity, respectively, pseudomonotonicity. It is well known, (see \emph{\cite{Gia-Mau}}), that monotonicity implies pseudomonotonicity, but the converse is not true.
\end{rmk}

\section{Existence of solutions for inverted problems}
 In this section we obtain an existence result for the inverted general variational inequality of Stampacchia type, $VI_{iS}(A,a,K)$.  Further, we establish a result, that may be viewed as a generalization of Minty's theorem for the inverted problems $VI_{iS}(A,a,K)$ and $VI_{iM}(A,a,K).$

Let $X$ and $Y$ be two Banach spaces. Recall that an operator $T:X\To Y$ is called weak to norm sequentially continuous at $x\in X$, if for every sequence $x_n$ that converges weakly to $x,$ we have that $T(x_n)\To T(x)$ in the topology of the norm of $Y.$ An operator $T:X\To Y$ is called weak to weak-sequentially continuous at $x\in X$, if for every sequence $x_n$ that converges weakly to $x,$ we have that $T(x_n)$ converges to $T(x)$ in the weak topology of $Y.$
We have the following existence result for the problem $VI_{iS}(A,a,K).$

\begin{thm}\label{t4.3.7} If $A$ is weak to norm sequentially continuous and is of type ql, $a$ is weak to weak sequentially continuous, and $K$ is weakly compact and convex, then the inverted general variational inequality of Stampacchia type $VI_{iS}(A,a,K)$ admits solutions.
\end{thm}
\begin{proof}
Considering the application $G:K\rightrightarrows K$, $G(y):=\{x\in K: \<A(y)-A(x),a(x)\>\ge 0\}$
the inverted general variational inequality of Stampacchia type $VI_{iS}(A,a,K)$ reduces to finding an element $x\in\bigcap_{y\in K}G(y).$

We show, that $G$ is a KKM mapping and the assumptions of Fan's lemma are satisfied.
Let $y_1,y_2,...,y_n\in K$ and $y\in co\{y_1,y_2,...,y_n\}.$ Let us suppose that $ y\not\in \bigcup_{i=1}^n G(y_i).$ Then,
 $\<A(y_1)-A(y),a(y)\><0$, $\<A(y_2)-A(y),a(y)\><0$,\ldots,$\<A(y_n)-A(y),a(y)\><0$.

Since $A$ is of type ql and $y\in co\{y_1,y_2,...,y_n\}$, according to Theorem \ref{t1}  we have
 $A(y)\in co\{A(y_1),A(y_2),...,A(y_n)\},$ hence, there exists $\l_i\ge 0,\, i=\overline{1,n}$ where $\sum_{i=1}^n\l_i=1,$ such that $ A(y)=\sum_{i=1}^n\l_i A(y_i).$

We have $\sum_{i=1}^n\l_i\<A(y_i)-A(y),a(y)\><0.$ On the other hand\\ $\sum_{i=1}^n\l_i\<A(y_i)-A(y),a(y)\>=\sum_{i=1}^n\<\l_i(A(y_i)-A(y)),a(y)\>=0,$ contradiction. Hence, $G$ is a KKM mapping.

We show next, that $G(y)$ is weakly compact for all $y\in K.$
Obviously $G(y)\neq\emptyset,$ since for all $y\in K$ we have $y\in G(y).$ For $y\in K$ consider a sequence $(x_n)_{n\in\Natural}\subseteq G(y)$ that converges weakly to $x\in K.$

 We have $\<A(y)-A(x_{n}),a(x_n)\>\ge 0$ for all $n\in \Natural.$ But, $\<A(y)-A(x_{n}),a(x_n)\>=\<A(y)-A(x),a(x_n)\>+\<A(x)-A(x_{n}),a(x_n)\>.$  Since $a$ is weak to weak sequentially continuous, obviously, $\<A(y)-A(x),a(x_n)\>\To \<A(y)-A(x),a(x)\>,\, n\To\infty.$ By the triangle inequality we get $|\<A(x)-A(x_{n}),a(x_n)\>|\le\|A(x)-A(x_{n})\|\|a(x_{n})\|.$ But, the sequence $(a(x_{n}))_{n\in\Natural}$  converges weakly to $a(x)$ and $x\in K,$ hence, $(a(x_{n})_{n\in\Natural}$  it is bounded.  Let $M>0,$ such that $\|a(x_{n})\|\le M,$ for all $n\in\Natural.$ It results that  $|\<A(x)-A(x_{n}),a(x_n)\>|\le M\|A(x)-A(x_{n})\|$ for all $n\in\Natural,$ and taking the limit $n\To\infty,$ and taking into account that $A$ is weak to norm-sequentially continuous, we get that $\<A(y)-A(x_{n}),a(x_n)\>\To 0,\,n\To\infty.$

Thus,  $\<A(y)-A(x_{n}),a(x_n)\>\To\<A(y)-A(x),a(x)\>,\, n\To\infty$, and from $\<A(y)-A(x_{n}),a(x_n)\>\ge 0$ for all $n\in \Natural,$ we get that $\<A(y)-A(x),a(x)\>\ge 0.$ Hence, $x\in G(y),$ which means that $G(y)$ is weakly sequentially closed for all $y\in K.$

 We prove next that $G(y)$ is weakly compact, for all $y\in K.$ For fixed $y$ let $(x_n)\subseteq G(y).$ Since $G(y)\subseteq K,$ we have $(x_n)\subseteq K.$ But $K$ is weakly compact, and according to Eberlein-\u{S}mulian theorem (see, for instance, \cite{fabian}), $K$ is weakly sequentially compact. Thus, there exists a subsequence $(x_{n_k})$ of $(x_n)$ convergent to a point $x\in K.$ But then, $(x_{n_k})\subseteq G(y),$ and from the weak sequentially closedness of $G(y)$ results that $x\in G(y).$ Hence, $G(y)$ is weakly sequentially compact, and according to Eberlein-\u{S}mulian theorem is weakly compact.
So we have that $G(y)$ is weakly compact for all $y\in K,$ which obviously implies that is weakly closed as well.

Hence, $G$ is a KKM mapping that satisfies the assumptions of Ky Fan's lemma.

Consequently, $\bigcap_{y\in K} G(y)\neq\emptyset,$ hence, there exists $x\in K$ such that $\<A(y)-A(x),a(x)\>\ge 0$ for all $y\in K.$
\end{proof}

\begin{rmk}\rm\label{r4.3.6} It can be shown in a similar way to the proof of Theorem $4.1$ from \cite{La}, that if $A$ is weak to weak sequentially continuous and is of type ql, $a$ is weak to norm sequentially continuous, and $K$ is weakly compact and convex, then $VI_{iS}(A,a,K)$ admits solutions. The next example shows, that without the assumption that the operator $A$ is of type ql, this conclusion, and also the conclusion of Theorem $\ref{t4.3.7}$, fails even in finite dimension.
\end{rmk}

\begin{ex}\rm\label{ex4.3.2} Let us consider the operators $a:K\To\Real^2,\, a(x,y)=(1,-x),$ and $\ds A:K\To K,\, A(x,y)=\left(x^2 y,xy\right),$ where $K=[-1,1]\times[-1,1]\subseteq\Real^2$. Then obviously  $a$ and $A$ are (norm to norm) continuous, $K$ is compact and convex but the inverted general variational inequality of Stampacchia type has no solutions.
\end{ex}

Indeed $A$ is not of type ql, since for $\ds\left(\frac12,\frac12\right)\in[(0,0),(1,1)]$ we have $\ds A\left(\frac12,\frac12\right)=\left(\frac18,\frac14\right)\not\in[(0,0),(1,1)]=[A(0,0),A(1,1)].$ Hence, all the assumptions of the Theorem \ref{t4.3.7} are satisfied excepting the one, that $A$ is of type ql.

 Let us suppose that there exists $(x,y)\in K,$ a solution of the problem $VI_{iS}(A,a,K).$ Then for every $(u,v)\in K$ we have $\<A(u,v)-A(x,y),a(x,y)\>\ge 0$ or equivalently $uv(u-x)\ge 0.$ For $u=v=-1$ we get $-1-x\ge 0,$ but $x\in[-1,1],$ hence $x=-1.$

 For $u=1,\,v=-1$ we get $x-1\ge 0,$ but $x\in[-1,1],$ hence $x=1.$ Contradiction.

In what follows, we establish some results for the problems $VI_{iS}(A,a,K)$ and $VI_{iM}(A,a,K)$, that may be viewed as generalizations of Minty's theorem.
Recall that Minty's theorem states, that if the operator $A:K\To X^*$ is hemicontinuous and monotone in Minty-Browder sense, then the solutions of Stampacchia variational inequality, $VI_S(A,K)$, and the solutions of Minty variational inequality, $VI_M(A,K)$, coincide (see \cite{Fe}). Actually, one needs to assume less.

\begin{thm}(Minty) Let $A:K \To X^* $ be an operator.
\begin{itemize}
\item[i)] If $A$ is hemicontinuous on $K$, and $K$ is convex, then every $x\in K$ which
solves $VI_M(A,K)$ is also a solution of $VI_S(A,K).$
\item[ii)] If, instead, $A$ is monotone on the convex set $K$,
then every $x\in K$ which solves $VI_S(A,K)$ is also a solution of $VI_M(A,K).$
\end{itemize}
\end{thm}

We have the following Minty type theorem.

\begin{thm}\label{t4.3.8} Let $K\subseteq X$ convex, and let $A:K\To X^*$ and $a:K\To X$ be two given operators.
 \begin{itemize}
 \item[i)]  If $A$ is monotone relative to $a$ and $x\in K$ is a solution of the inverted general variational inequality of Stampacchia type $VI_{iS}(A,a,K),$ then $x$ is a solution of the inverted general variational inequality of Minty type $VI_{iM}(A,a,K).$
 \item[ii)] If $A$ is of type strict ql and  $a$ is continuous on line segments, then every a solution of the inverted general variational inequality of Minty type $VI_{iM}(A,a,K),$ is also solution of the inverted general variational inequality of Stampacchia type $VI_{iS}(A,a,K).$
 \end{itemize}
\end{thm}
\begin{proof} "i)" Let $x\in K$ be a solution of the inverted general variational inequality of Stampacchia type $VI_{iS}(A,a,K).$ Then $\<A(y)-A(x),a(x)\>\ge 0,\,\mbox{\,for all\,}\, y\in K.$ Since $A$ is monotone relative to $a$ we have $\<A(y)-A(x),a(y)-a(x)\>\ge 0$ for all $y\in K$ or equivalently $\<A(y)-A(x),a(y)\>\ge\<A(y)-A(x),a(x)\>,$ hence  $x$ is a solution of the inverted general variational inequality of Minty type $VI_{iM}(A,a,K).$ It can be observed that we did not used the fact that $K$ is convex.

"ii)" Let $x\in K$ be a solution of the inverted general variational inequality of Minty type $VI_{iM}(A,a,K).$ Then $\<A(y)-A(x),a(y)\>\ge 0,\,\mbox{\,for all\,}\, y\in K.$ Let $z\in K.$ Since $K$ is convex, we have $y=x+t(z-x)\in K$ for all $t\in(0,1).$ Since $A$ is of type strict ql we have $A(y)\in(A(x),A(z))$ that is, $A(y)=A(x)+\a(A(z)-A(x))$ for some $\a\in(0,1).$ But $\<A(y)-A(x),a(y)\>\ge 0$ or equivalently $\<\a(A(z)-A(x)),a(x+t(z-x))\>\ge 0.$ By simplifying with $\a$ and taking the limit $t\downarrow 0,$ since $a$ is continuous on line segments we obtain $\<A(z)-A(x),a(x)\>\ge 0.$ But $z$ is arbitrary, hence $\<A(z)-A(x),a(x)\>\ge 0$ for all $z\in K,$ or in other words, $x$ is a solution of the inverted general variational inequality of Stampacchia type $VI_{iS}(A,a,K).$
\end{proof}
\begin{rmk}\rm It can be easily observed from the proof, that in Theorem \ref{t4.3.8} i) we can replace the assumption that the operator $A$ is monotone relative to $a$, with the assumption of a-pseudomonotonicity of $A.$
\end{rmk}
The condition that the operator $A$ is of type ql in the hypothesis of Theorem \ref{t4.3.8} ii) is essential as the next example shows.
\begin{ex}\rm\label{ex4.3.4} Let  us consider the functions $A:K\To[0,1]$ and $a:K\To K,$ $A(x)=\left\{
\begin{array}{lll}
-2x-1,\,\mbox{if } x\in\left[-1,-\ds\frac12\right],\\
2x+1,\,\mbox{if } x\in\left(-\ds\frac12,0\right],\\
-2x+1,\,\mbox{if } x\in\left(0,1\right],\\
\end{array}
\right.$ and
$ a(x)=\left\{
\begin{array}{ll}
-\ds\frac23 x+\frac13,\,\mbox{if } x\in\left[-1,\ds\frac12\right],\\
-2x+1,\,\mbox{if } x\in\left(\ds\frac12,1\right],\\
\end{array}
\right.$ where the set $K=[-1,1]$ is obviously convex.
\end{ex}
It can be easily checked that $A$ respective $a$ are continuous. We show that $A$ is not of type ql, and  $x_0=-\ds\frac12\in K$ is a solution of the inverted general variational inequality of Minty type, $VI_{iM}(A,a,K)$, but is not a solution of the inverted variational general inequality of Stampacchia type, $VI_{iS}(A,a,K).$

Indeed, since $A$ is not a monotone function, according to Proposition $3.2$ from \cite{La}, $A$ is not of type ql.
Easily can be verified that $\left(A(y)-A\left(x_0\right)\right)\cdot a(y)\ge 0$ for all $y\in K,$ hence $x_0$ is a solution of the inverted general variational inequality of Minty type, $VI_{iM}(A,a,K).$ On the other hand, for $y=\ds\frac34\in K$ we obtain $\left(A(y)-A\left(x_0\right)\right)\cdot a(x_0)<0$ 
   which shows that $x_0$ is not a solution of the inverted general variational inequality of Stampacchia type, $VI_{iS}(A,a,K)$.

 Even more, the condition that the operator $A$ is of type strict ql in the hypothesis of Theorem \ref{t4.3.8}  ii) is also essential.
 \begin{ex}\label{ex4.3.31}\rm Let $K=[-1,1],$ $A,a:K\To\Real,\, A(x)=\left\{\begin{array}{ll} -1,\,\mbox{if\,} x\in[-1,0),\\ 1,\,\mbox{if\,} x\in[0,1],\end{array}\right.\, \mbox{ and }$\\$a(x)=x.$
It can be easily verified that $K$ is convex, $a$ is continuous, $A$ is of type ql, but is not of type strict ql. We show that $x_0=\ds\frac12$ is a solution of $VI_{iM}(A,a,K)$ but is not solution of $VI_{iS}(A,a,K).$
 \end{ex}
 Indeed $(A(y)-A(x_0))\cdot a(y) =(A(y)-1)\cdot y=\left\{\begin{array}{ll} -2y,\,\mbox{if\,} y\in[-1,0),\\ 0,\,\mbox{if\,} y\in[0,1].\end{array}\right.$
Hence, $(A(y)-A(x_0))\cdot a(y)\ge0$ for all $y\in[-1,1]$, consequently  $x_0$ is a solution of $VI_M(A,a,K).$
On the other hand, for $y=-\ds\frac12,$ we have $(A(y)-A(x_0))\cdot a(x_0)<0$, consequently  $x_0$ is not a solution of $VI_S(A,a,K).$

\section{Applications}

In this section we prove Brouwer's fixed point theorem. Recall, that Brouwer's fixed point theorem states, that if $F:K\To K$ is a continuous function, where $K\subseteq\Real^n$ is a compact and convex, then $F$ admits a fixed point, that is, there exists $x\in K,$ such that $F(x)=x$ (see, for instance, \cite{Ist,Ru-Pe}).

Now, according to Theorem \ref{t4.3.7}, if $A$ is weak to norm sequentially continuous and is of type ql, $a$ is weak to weak sequentially continuous, and $K$ is weakly compact and convex, then the inverted general variational inequality of Stampacchia type $VI_{iS}(A,a,K)$ admits solutions. Let $K\subseteq \Real^n$ compact and convex, $A:K\To K,\, A\equiv id_K$ and $a:K\to\Real^n,\, a(x)=x-F(x).$ Obviously $A$ and $a$ are continuous. Hence, the assumptions of Theorem \ref{t4.3.7} are satisfied, consequently, there exists $x_0\in K$ such that $\<y-x_0,x_0-F(x_0)\>\ge 0,$ for all $y\in K.$  Since $Im(F)\subseteq K,$ for $y=F(x_0)\in K$ we obtain $\<F(x_0)-x_0,x_0-F(x_0)\>\ge 0.$ Hence, $-\|F(x_0)-x_0\|^2\ge 0,$ so we have $F(x_0)=x_0.$

\end{document}